\documentclass[a4paper,12pt]{article}
\usepackage{amssymb,amsmath,amsthm,latexsym}
\usepackage{amsfonts}
	\usepackage{amsfonts}
\usepackage{graphicx}
\usepackage[pdftex,bookmarks,colorlinks=false]{hyperref}
\usepackage{verbatim}
\usepackage{caption}
\usepackage{subcaption}

\usepackage{fancyhdr}
\newtheorem{theorem}{Theorem}[section]

\newtheorem{corollary}[theorem] {Corollary}
\newtheorem{definition}[theorem]{Definition}

\newtheorem{lemma} [theorem]{Lemma}

\newtheorem{proposition}[theorem]{Proposition}

\title{\bf A Study on Semi-arithmetic Integer Additive Set-Indexers of Graphs}
\author{{\bf N K Sudev \footnote{Department of Mathematics, Vidya Academy of Science \& Technology, Thalakkottukara, Thrissur - 680501, email: {\em sudevnk@gmail.com}}} and {\bf K A Germina\footnote{Department of Mathematics, School of Mathematical \& Physical Sciences, Central University of Kerala, Kasaragod, email:{\em srgerminaka@gmail.com}}}}
\date{}
\begin{document}
\maketitle

\begin{abstract}
An integer additive set-indexer is defined as an injective function $f:V(G)\rightarrow 2^{\mathbb{N}_0}$ such that the induced function $g_f:E(G) \rightarrow 2^{\mathbb{N}_0}$ defined by $g_f (uv) = f(u)+ f(v)$ is also injective. An integer additive set-indexer $f$ is said to be an arithmetic integer additive set-indexer if every element of $G$ are labeled by non-empty sets of non negative integers, which are in arithmetic progressions. An integer additive set-indexer $f$ is said to be a semi-arithmetic integer additive set-indexer if vertices of $G$ are labeled by non-empty sets of non negative integers, which are in arithmetic progressions, but edges are not labeled by non-empty sets of non negative integers, which are in arithmetic progressions. In this paper, we discuss about semi-arithmetic integer additive set-indexer and establish some results on this type of integer additive set-indexers. 
\end{abstract}

\section{Introduction}

\subsection{Preliminaries on Set-labeling}

For all  terms and definitions, not defined in this paper, we refer to \cite{FH} and for more about graph labeling, we refer to \cite{JAG}. Unless mentioned otherwise, all graphs considered here are simple, finite and have no isolated vertices. All sets mentioned in this paper are finite sets of non-negative integers. We denote the cardinality of a set $A$ by $|A|$.

Let $\mathbb{N}_0$ denote the set of all non-negative integers. For all $A, B \subseteq \mathbb{N}_0$, the sum of these sets,denoted by  $A+B$, is defined by $A + B = \{a+b: a \in A, b \in B\}$.  The set $A+B$ defined above is known as the {\em sum set} of the sets $A$ and $B$. 

The following are the major concepts introduced in \cite{GS0}.

Let $A$ and $B$ the set-labels of two adjacent vertices of a given graph $G$. Two ordered pairs $(a,b)$ and $(c,d)$ in $A\times B$ {\em compatible} if $a+b=c+d$. If $(a,b)$ and $(c,d)$ are compatible, then we write $(a,b)\sim (c,d)$. Clearly, $\sim$ is an equivalence relation. A {\em compatible class} of an ordered pair $(a,b)$ in $|A\times B|$ with respect to the integer $k=a+b$ is the subset of $A\times B$ defined by $\{(c,d)\in A\times B:(a,b)\sim (c,d)\}$ and is denoted by $[(a,b)]_k$ or $\mathsf{C}_k$.  If $(a,b)$ is the only element in the compatibility class $[(a,b)]_k$, then it is called a {\em trivial class}. The compatibility classes which contain the maximum possible number of elements is called {\em saturated classes}. The compatibility class that contains maximum elements is called a {\em maximal compatibility class}.
 
\begin{proposition}\label{P-CardCC}
\cite{GS0} The maximum possible cardinality of a compatibility class in $(A,B)$ is $n$, where $n=min(|A|,|B|)$. That is, the cardinality of a saturated class in $(A,B)$ is $min(|A|,|B|)$.
\end{proposition}

The number of distinct compatibility classes in $A\times B$ is called the {\em compatibility index} of the pair of sets $(A,B)$ and is denoted by $\mho_{(A,B)}$. 

\subsection{Integer Additive Set-Indexers}

An {\em integer additive set-indexer} (IASI, in short) is defined in \cite{GA} as an injective function $f:V(G)\rightarrow 2^{\mathbb{N}_0}$ such that the induced function $f^{+}:E(G) \rightarrow 2^{\mathbb{N}_0}$ defined by $f^{+} (uv) = f(u)+ f(v)$ is also injective.  A graph $G$ which admits an IASI is called an IASI graph.

The cardinality of the labeling set of an element (vertex or edge) of a graph $G$ is called the {\em set-indexing number} of that element.

\begin{lemma}\label{L-3}
\cite{GS0} Let $f$ be an IASI of a graph $G$ and $u,v$ be two vertices of $G$. Then, $f^{+}(uv)= f(u)+f(v)=\{a+b:a\in f(u), b\in f(v)\}$. Then, the set-indexing number of the edge $uv$ is $|f^{+}(uv)| = \mho_{(f(u),f(v))}$.
\end{lemma}

An IASI is said to be {\em $k$-uniform} if $|f^{+}(e)| = k$ for all $e\in E(G)$. That is, a connected graph $G$ is said to have a $k$-uniform IASI if all of its edges have the same set-indexing number $k$.  The vertex set $V$ of a graph $G$ is defined to be {\em $l$-uniformly set-indexed}, if all the vertices of $G$ have the set-indexing number $l$.


A {\em strong IASI} is defined in \cite{GS2} as an IASI $f$ such that $|f^{+}(uv)|=|f(u)| |f(v)|$ for all $u,v\in V(G)$. A graph which admits a  strong IASI may be called a {\em strong IASI graph}. A  strong  IASI is said to be  {\em strongly uniform IASI} if $|f^{+}(uv)|=k$, for all $u,v\in V(G)$ and for some positive integer $k$.

\subsection{Arithmetic Integer Additive Set-Indexers}

The studies about arithmetic IASIs of graphs, made in \cite{GS9}, \cite{GS10} and in \cite{GS11}, have established the following concepts.

As elements in the set-labels of all elements of $G$ are in arithmetic progression, they must contain at least three elements. By the term, an {\em arithmetically progressive set}, (AP-set, in short), we mean a set whose elements are in arithmetic progression. We call the common difference of the set-label of an element of a given graph, the {\em deterministic index} of that element.

Let $f:V(G)\to 2^{\mathbb{N}_0}$ be an IASI on $G$. For any vertex $v$ of $G$, if $f(v)$ is an AP-set, then $f$ is called a {\em vertex-aritmetic IASI} of $G$.A graph that admits a vertex-arithmetic IASI is called a {\em vertex-arithmetic IASI graph}. For an IASI $f$ of $G$, if $f^+(e)$ is an AP-set, for all $e\in E(G)$, then $f$ is called an {\em edge-aritmetic IASI} of $G$. A graph that admits an edge-arithmetic IASI is called an {\em edge-arithmetic IASI graph}.

An IASI is said to be an {\em arithmetic integer additive set-indexer} if it is both vertex-arithmetic and edge-arithmetic. That is, an arithmetic IASI is an IASI $f$, under which the set-labels of all elements of a given graph $G$ are AP-sets. A graph that admits an arithmetic IASI is called an {\em arithmetic IASI graph}. If all the set-labels of all vertices of a graph $G$ are AP-sets and the set-labels of edges are not AP-sets,  then the corresponding IASI is called {\em semi-arithmetic IASI}. If all the set-labels of all elements of a graph $G$ are AP-sets with the same difference $d$, then the corresponding IASI is called {\em isoarithmetic IASI}. An arithmetic IASI $f$ of a graph $G$, under which the differences $d_i$ and $d_j$ of the set-labels $f(v_i)$ and $f(v_j)$ respectively for two adjacent vertices $v_i$ and $v_j$ of $G$, holds the conditions $d_j=kd_i$ and $k$ is a non-negative integer such that $1< k \le |f(v_i)|$ is called {\em biarithmetic IASI}.

\begin{theorem}\label{T-AIASI-g}
\cite{GS9} A graph $G$ admits an arithmetic IASI graph $G$ if and only if for any two adjacent vertices in $G$, the deterministic index of one is a positive integral multiple of the deterministic index of the other and this integer is less than or equal to the cardinality of the set-label of the latter.
\end{theorem}

\section{Semi-Arithmetic IASIs}

\begin{definition}{\rm
A vertex-arithmetic IASI $f$ of a graph $G$, under which the differences $d_i$ and $d_j$ of the set-labels $f(v_i)$ and $f(v_j)$ respectively for two adjacent vertices $v_i$ and $v_j$ of $G$, holds the conditions $d_j=kd_i$ and $k$ is a non-negative integer greater than $|f(v_i)|$ is called the {\em semi-arithmetic IASI of the first kind}.}
\end{definition}

\begin{definition}{\rm
A vertex-arithmetic IASI $f$ of a graph $G$, under which the differences $d_i$ and $d_j$ of the set-labels $f(v_i)$ and $f(v_j)$ respectively for two adjacent vertices $v_i$ and $v_j$ of $G$ are not multiples of each other, is called the {\em semi-arithmetic IASI of the second kind}.}
\end{definition}

\begin{theorem}\label{T-SAIASI1}
Every first kind semi-arithmetic IASI of a graph $G$ is a strong IASI of $G$. 
\end{theorem}
\begin{proof}
Assume that the deterministic indices $d_i$ and $d_j$ of two adjacent vertices $v_i$ and $v_j$ respectively in $G$, where $d_i<d_j$ such that $d_j>|f(v_i)|.d_i$. Assume that $f(v_i)=\{a_r = a+rd_i:0 \le r <m\}$ and $f(v_j)=\{b_s=b+s\,k\,d_i:0\le s <n\}$. Now, arrange the terms of $f^+(v_iv_j)=f(v_i)+f(v_j)$ in rows and columns as follows. For $b_s\in f(v_j), 0\le s<n$, arrange the terms of $f(v_i)+b_s$ in $(s+1)$-th row in such a way that equal terms of different rows come in the same column of this arrangement. Then, the common difference between consecutive elements in each row is $d_i$. Since $k>|f(v_i)|$, the difference between the final element of any row and first element of its succeeding row is always greater than $d_i$. Therefore, no element is repeated in this arrangement. Therefore, total number of elements in $f(v_i)+f(v_j)$ is $|f(v_i)|\, |f(v_j)|$. Hence, $f$ is a strong IASI.
\end{proof}

Recall the following result proved in \cite{GS0}.

\begin{proposition}\label{T-SIASI1a}
\cite{GS0} If $f$ is strong IASI defined on a graph $G$, then for each adjacent pair of vertices $u$ and $v$ of $G$, each compatibility class of the pair of set-labels $f(u)$ and $f(v)$ is a trivial class. 
\end{proposition}

\begin{corollary}
Let $f$ be a first kind semi-arithmetic IASI of a graph $G$ and let $v_i$ and $v_j$ be two adjacent vertices in $G$. Then,all the compatibility classes in $f(v_i)\times f(v_j)$ are trivial classes. 
\end{corollary}
\begin{proof}
Since $f$ is a first kind semi-arithmetic IASI by Theorem \ref{T-SAIASI1}, $f$ is a strong IASI. Then, by Proposition \ref{T-SIASI1a}, all compatibility classes in $f(v_i)\times f(v_j)$ are trivial classes.
\end{proof}

An interesting question that arises in this context is about the existence of uniform semi-arithmetic IASIs. The following theorem establishes the necessary and sufficient condition for a semi-arithmetic IASI to be uniform.

\begin{proposition}\label{P-SAIASI2}
If $f$ is a first kind semi-arithmetic IASI of a graph $G$, then no edge of $G$ has a prime set-indexing number.
\end{proposition}
\begin{proof}
Let $f$ be a first kind semi-arithmetic IASI of a graph $G$. Then, by Theorem \ref{T-SAIASI1}, $f$ is a strong IASI. Therefore, for any two adjacent vertices $u$ and $v$ of $G$, $|f^+(uv)|=|f(u)|\,|f(v)|$. If the edge $uv$ has a prime set-indexing number, say $p$, then $|f(u)|$ and $|f(v)|$ divide $p$. Therefore, either $|f(u)|=1$ or $|f(v)|=1$, which is a contradiction to the fact that every set-label have at least three elements. Hence, no edge of $G$ can have a prime set-indexing number.
\end{proof}

\begin{theorem}\label{T-SAIASI3}
A first kind semi-arithmetic IASI of a graph $G$ is a uniform IASI if and only if either $G$ is bipartite or $V(G)$ is uniformly set-indexed.
\end{theorem}
\begin{proof}
Let $f$ be a first kind semi-arithmetic IASI defined on a graph $G$. For a positive integer $l$, assume that $f$ is an $l$-uniform IASI. Let $v_i$ and $v_j$ be any two adjacent vertices of $G$ such that $|f(v_i)|=m$ and $|f(v_j)|=n$. Then, $m\,n =l$. Since, $f$ is $l$-uniform, every vertex that is adjacent to the vertex $v_i$ must have the set-indexing number $n$ and every vertex that is adjacent to the vertex $v_j$ must have the set-indexing number $m$. That is, in general, all the vertices adjacent to a vertex having set-indexing number $m$, must have the set-indexing number $n$ and all the vertices adjacent to a vertex having set-indexing number $n$, must have the set-indexing number $m$. If $m=n$, then our proof is complete. If $m\ne n$, then let $X$ be the set of all vertices of $G$ having set-indexing number $m$ and $Y$ be the set of all vertices of $G$ having set-indexing number $n$. Since, $m^2\ne l$, no two vertices in $X$ can be adjacent. Similarly, since $n^2\ne l$, no two vertices in $Y$ also can be adjacent. Therefore, $(X,Y)$ is a bipartition of $G$.

Conversely, assume that either $G$ is bipartite or $V(G)$ is uniformly set-indexed. If $V(G)$ is $n$-uniformly indexed, then $|f^+(uv)|=|f(u)|\,|f(v)|=n^2 ~ \forall u,v~\in V(G)$. That is, $f$ is $n^2$-uniform. 
Now assume that $V(G)$ is not uniformly set-indexed. Then, by hypothesis, $G$ is bipartite. Let $(X,Y)$ be a bipartition of $G$. For a positive integer $d$, label all the vertices in $X$ by distinct $m$-element AP-sets with common difference $d$,  and label all the vertices in $Y$ by distinct $n$-element AP-sets with common difference $k\,d$, where $k$ is a positive integer such that $k> max_{v_i\in X}\{f(v_i)\}$. Then, $f$ is a first kind semi-arithmetic IASI and by Theorem \ref{T-SAIASI1}, is a strong IASI. Therefore, every edge of $G$ has the set-indexing number $mn$.
\end{proof}

What is condition required for a second kind semi-arithmetic IASI to be a strong IASI? The following theorem provides an answer to this question.

\begin{theorem}
Let $f$ be an semi-arithmetic IASI defined on $G$.  Also, let $|f(v_j)|=q.|f(v_i)|+r, 0<r<|f(v_i)|$. Then, $f$ is a strong IASI if and only if $q>|f(v_i)|$ or the differences$d_i$ and $d_j$ of two set labels $f(v_i)$ and $f(v_j)$ respectively, are relatively prime.
\end{theorem}\label{T-AIASI-A}
\begin{proof}
First assume that $q>|f(v_i)|$. We arrange the elements of $f(v_i)+f(v_j)$ into rows and columns such that the sum of elements of $f(v_i)$ with the $r$-th element of $f(v_j)$, $1\le r\le |f(v_j)|$, as the elements of $r$-th row of the new arrangement. Since $q>|f(v_i)|$, in this arrangements, all the elements in the $r+1$-th row will be greater than all elements of the $r$-th row. That is, all elements in this arrangement are distinct. Hence, $f$ is a strong IASI. 

Now, assume that $q\le |f(v_i)|$. Then, by hypothesis, $gcd (d_i,d_j)=1$. Therefore, $r$ can not be a divisor of $d_i$. Then, no two elements of $f(v_i)+f(v_j)$ can belong to the same compatible class. Hence, $f$ is a strong IASI.

Conversely, assume that $f$ is a strong IASI of $G$. Then, every compatible class $\mathsf{C}_{(a,b)}$ in $f(v_i)\times f(v_j)$ is a trivial class. This condition holds when $q>|f(v_i)$. Let $q\le |f(v_i)$. If $gcd(d_i,d_j)=t\ne 1$, then $t|d_i$ and hence $t|r$. Therefore, for some integers $q_1, q_2$, we have $q_1.d_i=q_2.r,~ q_1<q_2$. Hence, some terms in $f(v_i)\times f(v_j)$ are the same, which is a contradiction to the fact that $f$ is a strong IASI. Hence, $d_i$ (or $d_j$) is a multiple of $r$. Hence, $gcd (d_i,d_j)=1$.
\end{proof}

We note that an arithmetic IASI with arbitrary differences do not have saturated classes. In the following discussion, we find the number of maximal compatible classes for a second kind semi-arithmetic IASIs in the following theorem.

\begin{theorem}
Let $f$ be an arithmetic IASI with arbitrary differences on a graph $G$. Let $|f(v_j)|=q.|f(v_i)|+r$. Also, let $q_1$ and $q_2$ be the positive integers such that $q_1.|f(v_j)=q_2.r$. Then, the number of elements in a maximal compatible class of $f(v_i)\times f(v_j)$ is $\lfloor \frac{|f(v_j)|}{q_1} \rfloor$.
\end{theorem}
\begin{proof}
We use the same notations as in Theorem \ref{T-AIASI-A}. Let $|f(v_i)|=m$ and $|f(v_j)=n$. Arrange the elements of $f(v_i)+f(v_j)$ into rows and columns such that the sum of elements of $f(v_i)$ with the $r$-th element of $f(v_j)$, $1\le r\le |f(v_j)|$, as the elements of $r$-th row of the new arrangement. By Theorem \ref{T-AIASI-A}, a compatibility class contains two or more elements if $q\le |f(v_i)|$ and $gcd(d_i,d_j)\neq 1$. Hence, there exist some positive integers $q_1$ and $r_1$ such that $q_1.d_i=r_1.r,~ q_1<r_1$. If $q_1<n$, then some values appear in the arrangement $\lfloor \frac{n}{q_1} \rfloor$ times. 
\end{proof}

Analogous to the corresponding theorems for other types of IASI graphs, we observe the following result.

\begin{proposition}
Any subgraph of a semi-arithmetic IASI graph admits a semi-arithmetic IASI.
\end{proposition}

\begin{definition}{\rm
\cite{DBW} For a given graph $G$, its line graph $L(G)$ is a graph such that  each vertex of $L(G)$ represents an edge of $G$ and two vertices of $L(G)$ are adjacent if and only if their corresponding edges in $G$ incident on a common vertex in $G$.}
\end{definition}

\begin{proposition}
The line graph $L(G)$ of a semi-arithmetic graph never admits a semi-arithmetic IASI.
\end{proposition}
\begin{proof}
The set-labels of edges of a semi-arithmetic IASI graph $G$ are not AP-sets. Hence, the vertices in $L(G)$ corresponding to the edges in $G$ do not have AP-sets as their set-labels. Therefore, $L(G)$ does not admit a semi-arithmetic IASI.
\end{proof}

\begin{definition}{\rm
\cite{MB} The {\em total graph} of a graph $G$ is the graph, denoted by $T(G)$, is the graph having the property that a one-to one correspondence can be defined between its points and the elements (vertices and edges) of $G$ such that two points of $T(G)$ are adjacent if and only if the corresponding elements of $G$ are adjacent (either  if both elements are edges or if both elements are vertices) or they are incident (if one element is an edge and the other is a vertex). }
\end{definition}

\begin{proposition}
The total graph $T(G)$ of a semi-arithmetic graph never admits a semi-arithmetic IASI.
\end{proposition}
\begin{proof}
The set-labels of edges of a semi-arithmetic IASI graph $G$ are not AP-sets. Therefore, the vertices in $T(G)$ corresponding to the edges in $G$ do not have AP-sets as their set-labels. Hence, $T(G)$ does not admit a semi-arithmetic IASI.
\end{proof}

\begin{definition}{\rm
\cite{FH} By {\em edge contraction operation} in $G$, we mean an edge, say $e$, is removed and its two incident vertices, $u$ and $v$, are merged into a new vertex $w$, where the edges incident to $w$ each correspond to an edge incident to either $u$ or $v$.}
\end{definition}

\begin{proposition}
A graph obtained by contracting an edge of a semi-arithmetic IASI graph $G$ does not admit a semi-arithmetic IASI.
\end{proposition}
\begin{proof}
Let $e=uv$ be an arbitrary edge of a semi-arithmetic IASI graph $G$. Let $G'=G\circ e$. Let $w$ be the new vertex obtained by removing the edge $e$ and identifying the two vertices $u$ and $v$ to get a new vertex $w$. It is customary to assign the same set-label of $e$ to the new vertex $w$. Therefore, the set-label of $w$ is not an AP-set. Hence, $G'$ does not admit a semi-arithmetic IASI. 
\end{proof}

\begin{definition}{\rm
\cite{KDJ} Let $G$ be a connected graph and let $v$ be a vertex of $G$ with $d(v)=2$. Then, $v$ is adjacent to two vertices $u$ and $w$ in $G$. If $u$ and $v$ are non-adjacent vertices in $G$, then delete $v$ from $G$ and add the edge $uw$ to $G-\{v\}$. This operation is known as an {\em elementary topological reduction} on $G$.}
\end{definition}

\begin{proposition}
Let $G$ be a semi-arithmetic IASI graph and let $v$ be an arbitrary vertex of $G$ with $d(v)=2$ not contained in any triangle of $G$. Let $G'=(G-v)\{uw\}$, where $u$ and $w$ are adjacent vertices of $v$ in $G$. Then, $G'$ admits a semi-arithmetic IASI if and only if the deterministic indices of one of $u$ or $w$ is a positive integer multiple of the deterministic index of the other, where this integer is greater than the cardinality of the latter.
\end{proposition}
\begin{proof}
Let $v$ be an arbitrary vertex of $G$ with $d(v)=2$ not contained in any triangle of $G$. Since $d(v)=2$, $v$ is adjacent to two vertices, say $u$ and $w$ in $G$. Now, remove the veretx $v$ from $G$. Then , the edges $uv$ and $vw$ will be eliminated. Join the end vertices $u$ and $w$. Let $G'=(G-v)\{uw\}$.

Since, $G$ is semi-arithmetic, $G'$ is semi-arithmetic if and only if the set-label of the edge $uw$ is not an AP-set. This is possible only when the deterministic indices of one of $u$ or $w$ is a positive integer multiple of the deterministic index of the other, where this integer is greater than the cardinality of the latter. Therefore, $G'$ is semi-arithmetic if and only if the deterministic indices of one of $u$ or $w$ is a positive integer multiple of the deterministic index of the other, where this integer is greater than the cardinality of the latter.
\end{proof}

\begin{definition}{\rm
\cite{RJT} A {\em subdivision} of a graph $G$ is the graph obtained by adding vertices of degree two into its edges.}
\end{definition}

\begin{proposition}
A subdivision of a semi-arithmetic graph $G$ does not admit a semi-arithmetic IASI.
\end{proposition}
\begin{proof}
Let $e=uv$ be an arbitrary edge of a semi-arithmetic IASI graph $G$. Let $G'$ be a subdivision of $G$ obtained by a new vertex $w$ to the edge $e$. Therefore, the edge $uv$ will be replaced by the two edges $uw$ and $vw$ in $G$. It is customary to assign the same set-label of $uv$ to the new vertex $w$. Therefore, the set-label of $w$ is not an AP-set. Hence, $G'$ does not admit a semi-arithmetic IASI. 
\end{proof}

\section{Conclusion}

In this paper, we have discussed some characteristics of graphs which admit a certain type of IASI called semi-arithmetic IASI. We have formulated some conditions for some graph classes to admit semi-arithmetic IASIs and we have discussed about some characteristics of semi-arithmetic IASI graphs. Certain problems in this area are still open. 

The IASIs under which the vertices of a given graph are labeled by different standard sequences of non negative integers, are also worth studying.   The problems of establishing the necessary and sufficient conditions for various graphs and graph classes to have certain IASIs still remain unsettled. All these facts highlight a wide scope for further studies in this area.

\end{document}